\documentclass[11pt]
{article}
\usepackage{amsmath,amssymb,amsthm
,makeidx,fancyhdr
}

\usepackage[utf8]{inputenc}

\usepackage{color}

\theoremstyle{definition}
\newtheorem{definition}{Definition}

\newtheorem{observation}[definition]{Observation}
\newtheorem{question}[definition]{Question}

\theoremstyle{plain}
\newtheorem{theorem}[definition]{Theorem}
\newtheorem{proposition}[definition]{Proposition}
\newtheorem{lemma}[definition]{Lemma}

\newtheorem{corollary}[definition]{Corollary}

\newtheorem{remark}[definition]{Remark}

\newcommand{\mo}{\triangleleft}

\newcommand{\fr}{{}^\frown}

\newcommand{\name}{\dot}
\newcommand{\can}{\check}
\newcommand{\force}{\Vdash}

\newcommand{\la}{\langle}
\newcommand{\ra}{\rangle}
\newcommand{\elem}{\prec}
\newcommand{\uhr}{\restriction}
\newcommand{\dhr}{\downharpoonright}

\newcommand{\power}{\mathcal{P}}

\newcommand{\oln}{\overline{\nu}}
\newcommand{\olm}{\overline{\mu}}

\newcommand{\R}{\mathbb{R}}

\newcommand{\rad}{\mathbb{R}}

\newcommand{\U}{\vec{U}}
\newcommand{\F}{\mathcal{F}}

\newcommand{\MS}{\mathcal{MS}}
\newcommand{\MF}{\mathfrak{MF}}

\newcommand{\e}{\vec{e}}

\DeclareMathOperator{\cf}{cf}

\DeclareMathOperator{\Cof}{cf}
\DeclareMathOperator{\otp}{otp}
\DeclareMathOperator{\supp}{supp}
\DeclareMathOperator{\suc}{succ}

\DeclareMathOperator{\RP}{RP}

\DeclareMathOperator{\WRP}{WRP}
\DeclareMathOperator{\LRP}{LRP}
\DeclareMathOperator{\Reg}{Reg}

\title{Diamonds, Compactness, and Measure Sequences}
\author{Omer Ben-Neria}
\date{}

\begin{document}
\maketitle
\begin{abstract}
We establish the consistency of the failure of the diamond principle on a cardinal $\kappa$ which satisfies a strong simultaneous reflection property. 
The result is based on an analysis of Radin forcing, and further leads to a characterization of weak compactness of $\kappa$ in a Radin generic extension.
\end{abstract}

\section{Introduction}
In pursuit of an understanding of the relations between compactness and approximation principles, 
we address the following question: To what extent do compactness principles assert the existence of a diamond sequence?\\
The compactness principles considered in this paper are stationary reflection and weak compactness. The main result of this paper shows that a strong form of stationary reflection does not imply $\Diamond_\kappa$.
\begin{theorem}\label{Theorem - Main}
It is consistent relative to a certain hypermeasurability large cardinal assumption that there exists a cardinal $\kappa$ satisfying the following properties.
\begin{enumerate}
\item For every sequence $\vec{S} = \la S_i \mid i < \kappa\ra$ of stationary subsets in $\kappa$, there exists $\delta < \kappa$ such that all sets in $\vec{S}\uhr \delta =\la S_i  \mid i < \delta\ra$ reflect at $\delta$. 
\item $\Diamond_\kappa$ fails. 
\end{enumerate}
\end{theorem}

Let us recall the relevant definitions. Suppose that $\kappa$ is a regular cardinal. 
The diamond principle at $\Diamond_\kappa$, introduced by Jensen in \cite{JensenFS}, asserts the existence of a sequence $\la s_\alpha \mid \alpha < \kappa\ra$ of sets $s_\alpha \subset \alpha$, such that for every $X \subset \kappa$  the 
set $\{ \alpha < \kappa \mid s_\alpha = X \cap \alpha\}$ is stationary in $\kappa$. 
\\
We say that a stationary subset $S$ of $\kappa$ reflects at $\delta < \kappa$ if $S \cap \delta$ is stationary in $\delta$.  
A cardinal $\kappa$ is reflecting 
if every stationary subset of $\kappa$ reflects at some $\delta  <\kappa$. 
The stronger reflection property in the statement of Theorem \ref{Theorem - Main} will be called a strong simultaneous reflection property. It clearly implies that every family of less than $\kappa$ many stationary subsets of $\kappa$ simultaneously reflect at some $\delta < \kappa$. \\
Reflecting is a compactness type property \footnote{I.e., its contrapositive postulates that if $A$ is a subset of $\kappa$ and $A \cap \alpha$ is non-stationary for each $\alpha$, then $A$ is not stationary in $\kappa$.}. It well-known that a reflecting cardinal is greatly Mahlo, and that every weakly compact cardinal satisfies the strong simultaneous reflection property.
Although reflection is a consequence of the strong simultaneous reflection property,  the two properties may coincide: Jensen \cite{JensenFS} has shown that in $L$, a reflecting cardinal is weakly compact, and therefore satisfies the strong principle. In contrast, 
Harrington and Shelah \cite{HarringtonShelahSR} proved that the existence of a Mahlo cardinal is equi-consistent with reflection of every stationary subset of $\omega_2 \cap \Cof(\omega)$, while Magidor \cite{MagidorSR} proved  the stronger simultaneous reflection property for stationary subsets of $\omega_2 \cap \Cof(\omega)$ is equi-consistent with the existence of a weakly compact cardinal.\\
The relations between compactness (large cardinal) axioms and $\Diamond$ type principles have been extensively studied. 
See \cite{KanamoriDLU} for a comprehensive discussion of the problem. 
It is well-known that every measurable cardinal $\kappa$ carries a $\Diamond_\kappa$ sequence, and Jensen and Kunen \cite{JensenKunenLC} showed $\Diamond_\kappa$ holds at every subtle cardinal $\kappa$. In fact, they proved that a subtle cardinal $\kappa$ satisfies the stronger approximation property - $\Diamond_\kappa(\Reg)$. A $\Diamond_\kappa(\Reg)$ sequence is a diamond sequence which approximates subsets $X \subset \kappa$ on the (tighter set) of regular cardinals $\alpha < \kappa$ 
Nevertheless, not every large cardinal assumption implies the existence of a diamond sequence.
 Woodin first showed the stronger principle $\Diamond_\kappa(\Reg)$ can fail at a weakly compact cardinal. The result was extended by Hauser \cite{HauserDI} to indescribable cardinals, and by D{\u{z}}amonja and Hamkins \cite{DH} to strongly unfoldable cardinals.
Each of these results is established from the minimal relevant large cardinal assumption \footnote{E.g., the existence of a weakly compact cardinal $\kappa$ with $\neg\Diamond_\kappa(\Reg)$ is equi-consistent with the existence of a weakly compact cardinal.}, which are all compatible with $V = L$, and have been shown to be insufficient for establishing the violation of the full diamond principle:
Jensen \cite{JensenDM} has shown  $\neg\Diamond_\kappa$ at a Mahlo cardinal $\kappa$ implies the existence of $0^{\#}$, and Zeman \cite{ZemanDM} improved the lower bound of the assumption to the existence of an inner model $K$ with a Mahlo cardinal $\kappa$, such that for every $\epsilon < \kappa$, the set $\{ \alpha < \kappa \mid o^{K}(\alpha) \geq  \epsilon\}$ is stationary in $\kappa$. \\ 
Zeman's argument indicates that establishing the consistency of $\neg\Diamond_\kappa$ via forcing, requires changing the cofinality of many cardinals below $\kappa$.
Indeed,   starting from certain hypermeasurability (large cardinal) assumptions,
 Woodin \cite{CummingsWD} has shown $\neg\Diamond_\kappa$ is consistent with $\kappa$ being  inaccessible, Mahlo, or greatly Mahlo cardinal.
 Woodin's argument is based on Radin forcing $\R(\U)$ introduced in \cite{Radin}, which adds a closed unbounded subset to $\kappa$ consisting of indiscernibles associated with ultrafilters on $\kappa$ from a measure sequence $\U$. Theorem \ref{Theorem - Main}
 is based on Woodin's strategy, and relies on an analysis of Radin forcing. The analysis also leads to a characterization of weak compactness of $\kappa$ in a generic extension by $\R(\U)$. \\
Extending Woodin's result in a different direction, 
it has been recently shown in \cite{BGH} that the weak diamond principle,  $\Phi_\kappa$, also fails in Woodin's model. All the results of this paper concerning the failure of $\Diamond_\kappa$ are compatible with the argument for $\neg\Phi_\kappa$. 
In in another direction,  Golshani \cite{Golshani} has recently shown 
$\neg\Phi_\kappa$  is consistent with $\kappa$ being the first inaccessible cardinal.\\
%

A brief summary of this paper. The rest of this Section is devoted to reviewing Radin forcing $\R(\U)$ and its basic properties. 
Section \ref{Section 2} is devoted to studying the ground model sets $A \subset \kappa$ which remain stationary in a Radin generic extension. 
In Section \ref{Section 3} we extend the analysis to arbitrary stationary subsets of $\kappa$ in a generic extension, and prove Theorem \ref{Theorem - Main}.
Finally, in Section \ref{Section 4}, we go beyond reflection and consider the weak compactness of $\kappa$. We introduce a property of a measure sequence $\U$ called the weak repeat property (WRP), and prove $\kappa$ is weakly compact in a $\R(\U)$ extension if and only if $\U$ satisfies $\WRP$. 

\subsection{Additional information - Radin forcing}\label{Intro - Radin}
We review Radin forcing and its basic properties. Our presentation follows Gitik's Handbook chapter \cite{GitikHB}.  Thus, everything in Section \ref{Intro - Radin} (except Proposition \ref{Proposition - fresh sets}) can be found in \cite{GitikHB}.
Also, we shall follow the Jerusalem forcing convention of \cite{GitikHB}, that a condition $p$ is stronger (more informative) than a condition $q$ is denoted by $p \geq q$.

\begin{definition}
Let $\kappa$ be a measurable cardinal and $\U = \la \kappa\ra \fr \la U_\alpha \mid \alpha < \ell(\U)\ra$ be a sequence such that each $U_\alpha$ is a measure on $V_\kappa$ (i.e., a $\kappa$-complete normal ultrafilter on $V_\kappa$). 
For each $\beta  < \ell(\U)$, let $\U\uhr \beta$ denote the initial segment $\la \kappa \ra \fr \la U_\alpha \mid \alpha < \beta\ra$. 
 We say $\U$ is a \textbf{measure sequence} on $\kappa$ if there exists an elementary embedding $j : V \to M$ such that for each $\beta < \ell(\U)$, $\U\uhr \beta \in M$ and $U_\beta = \{ X \subset V_\kappa \mid \U\uhr \beta \in j(X)\}$. 
\end{definition}

We will frequently use the following notations. Let $\cap \U$ denote the filter $\bigcap_{\alpha  < \ell(\U)}U_\alpha$, and $\MS$ denote the set of measure sequences $\olm$ on measurable cardinals below $\kappa$. $\olm$ is of the form $\la \nu \ra  \fr \la u_i \mid i < \ell(\olm)\ra$, where each $u_i$ is a measure on $V_{\nu}$. We denote $\nu$ by $\kappa(\olm)$,
and $\cap \{u_i \mid i < \ell(\olm)\}$ by $\cap \olm$. 
For each $i < \ell(\olm)$, we denote $u_i$ by $\olm(i)$. \\

Let $\U$ be a measure sequence on $\kappa$. We proceed to define the Radin forcing $\R(\U)$. 
Define first a sequence of sets $A^n \subset \MS$, $n < \omega$. Let 
$A^0 = \MS$, and for every $n < \omega$, $A^{n+1} = \{\olm \in A^{n} \mid A^n \cap V_{\kappa(\olm)} \in \cap \olm\}$. 
Finally, set $\bar{A} = \bigcap_n A^n$. Using the embedding $j$ and the definition of the measures $U_\alpha$, $\alpha < \ell(\U)$, it is straightforward to verify $\bar{A} \in \bigcap \U$. 

\begin{definition}[Radin forcing]
$\R(\vec{U})$ consists of finite sequences $p = \la d_i \mid i \leq k\ra$ satisfying the following conditions.
\begin{enumerate}
\item 
For every $i \leq k$, $d_i$ is either of the form $\la \kappa_i \ra$ for some $\kappa_i < \kappa$,
or of the form $d_i = \la \vec{\mu_i} , a_i\ra$ where $\vec{\mu_i}$ is a measure sequence on a measurable cardinal $\kappa_i = \kappa(\vec{\mu_i}) \leq \kappa$, and $a_i \in \cap \vec{\mu_i}$ is a subset of $(\bar{A} \cap V_{\kappa_i}) \setminus V_{\kappa_{i-1}}$.
\item $\la \kappa_i \mid i \leq k\ra$ is a strictly increasing sequence and $\kappa_k = \kappa$.  
\item $d_k = \la \U, A\ra$ and $A \subset \bar{A}$. 
\end{enumerate}
For each $i \leq k$ we denote $\kappa_i$ by $\kappa(d_i)$, $\olm_i$ by $\olm(d_i)$, and $a_i$ by $a(d_i)$. 
For $m < k$, we denote $p^{\leq m} = \la d_i \mid i \leq m\ra$ and $p^{>m} = \la d_i \mid m < i \leq k\ra$. 
Given a condition $p = \la d_i \mid i \leq k\ra$, we frequently separate the top part $d_k = \la \vec{U} , A\ra$ from 
rest, and write $p = \vec{d}  \fr \la \vec{U}, A\ra$ or $p = p_0 \fr \la \vec{U},A\ra$, 
where $p_0 = \vec{d} = \la d_i \mid i < k\ra$. We denote the set of all lower parts of conditions by $\R_{<\kappa}$.
A condition $p^* =  \la d^*_i \mid i \leq k^*\ra$ is a \textbf{direct extension} of $p = \la d_i \mid i \leq k\ra$ if $k^* = k$ and $a(d^*_i) \subset a(d_i)$ whenever $a(d_i)$ exists.
A condition $p'$ is a \textbf{one point extension} of $p$ if there exists $j \leq k$ and a measure sequence $\vec{\nu} \in a(d_j)$ 
such that $p' = p \fr \la \oln \ra$ is either $\la d_i \mid i < j\ra \fr \la \vec{\nu} \ra \fr \la d_i \mid i \geq j\ra$ if $\vec{\nu} = \alpha$ is an ordinal, or $\la d_i \mid i < j\ra \fr \la \vec{\nu}, a(d_j) \cap V_{\kappa(\vec{\nu})}\ra \fr \la d_i \mid i \geq j\ra$ if $\vec{\nu}$ is a nontrivial measure sequence\footnote{Note that implicitly, we are assuming here that $\kappa(\vec{\nu}) > \kappa(d_{j-1})$ and $a(d_j) \cap V_{\kappa(\vec{\nu})} \in \cap \vec{\nu}$.}.
Let $p,\bar{p}$ be two conditions of $\R(\U)$ and $n < \omega$. We say that $\bar{p}$ is an $n$-extension of $p$, if there exists a sequence $\vec{\eta} = \la \oln_1, \dots , \oln_n\ra \subset \MS$ such that
$\bar{p} = ( \dots ((p \fr \la \oln_1\ra) \fr \la \oln_2 \ra ) \dots ) \fr \la \oln_n\ra$. We denote $\bar{p}$ by $p \fr \vec{\eta}$. 
Given two conditions $p,q \in \R(\U)$ we say that $q$ \textbf{extends} $p$, denoted $q \geq p$, if it is obtained from $p$ by a finite sequence of one point extensions and direct extensions. Equivalently, 
$q$ extends $p$ if there exists a finite sequence $\vec{\eta}$ so that  $q$ is a direct extension of $p \fr \vec{\eta}$. 
\end{definition}

\begin{definition}
Suppose that $ G \subset \R(\U)$ is a $V$ generic filter.
Define $\MS_G =  \{ \olm \in \MS \mid \exists p \in G, p = \la d_i \mid i \leq k\ra  \text{ and } \olm = \olm(d_i) \text{ for some } i < k\}$, and $C_G = \{ \kappa(\olm) \mid \olm \in \MS_G\}$. 
\end{definition}

A standard density argument shows that $\MS_G$  is almost contained in every ground model set $A \in \bigcap \U$ and it completely determines $G$. In particular, $V[G] = V[\MS_G]$. 
$C_G$ is called the generic Radin closed unbounded set. 
The definition of the forcing implies that if $p = \vec{d} \fr \la \U,A\ra$ and $q = \vec{e} \fr \la \U,B\ra$ are two conditions in $\R(\U)$ satisfying $|\vec{d}| = |\vec{e}|$ and $\olm(d_i) = \olm(e_i)$ for each $i < |\vec{d}|$, then $p,q$ are compatible. Since $|V_\kappa| = \kappa$, it follows $\R(\U)$ satisfies $\kappa^+$.c.c

\begin{lemma}${}$\\
\textbf{1.} 
$(\R(\U), \leq,\leq^*)$ satisfies the Prikry condition. Namely, for every condition $p \in \R(\U)$ and a statement of the forcing language $\sigma$, there exists $p^* \geq^* p$ which decides $\sigma$. \\
\textbf{2.} For each $p  = \vec{d} \fr \la \U,A\ra \in \R(\U)$ and $m < |\vec{d}|$, the forcing $\R(\U)/p$ is isomorphic to the product $\R(\olm(d_m))/p^{\leq m} \times \R(\U)/p^{>m}$.\\
\textbf{3.} For every condition $p  = \vec{d} \fr \la \U,A\ra \in \R(\U)$ and $m < |\vec{d}|$, the direct extension order $\leq^*$ of $\R(\U)/p^{>m}$ is $\kappa_m^+$-closed.
\end{lemma}
Combining the last two Lemmata with a standard factorization argument, it is routine to verify $\R(\U)$ preserves all cardinals. 
Further analysis of $\R(\U)$ relies on the notion of fat trees. 
\begin{definition}\label{Definition - fat trees}
Let $\olm$ be a measure sequence on a cardinal $\nu = \kappa(\olm)$. 
A tree $T \subset [V_\nu]^{\leq n}$, for some $n < \omega$,  is called $\olm$-\textbf{fat} if
it consists of sequences of measure sequences $\vec{\oln} =\la \oln_1,\dots,\oln_k\ra$, $k \leq n$, satisfying the following two conditions. 
\begin{enumerate}
\item $\kappa(\oln_1) < \dots < \kappa(\oln_k)$.
\item if $k < n$ then there exists some $i < \ell(\olm)$ so that the set $\suc_T(\vec{\oln}) = \{ \oln' \mid \vec{\oln} \fr \la \oln'\ra \in T\}$ belongs to $\olm(i)$. 
\end{enumerate} 
\end{definition}

Let $p \in R(\U)$ and $\vec{\eta} = \la \oln_1,\dots,\oln_n\ra$ be a sequence of measure sequences such that $p \fr \vec{\eta} \geq p$. 
We say that a sequence of sets $\vec{A} = \la A_1,\dots, A_n\ra$ is a $\vec{\eta}$-measure-one sequence, if for every $1 \leq l \leq n$, such that $\ell(\oln_l) > 0$, then $A_l \in \cap \oln_l$. 
Let $ p \fr \la \vec{\eta},\vec{A}\ra$ be the direct extension of $p' = p \fr \vec{\eta}$ obtained by intersecting $A_l$ with the measure-one set $a^{p'}(\oln_l)$ appearing in $p'$, for every $l \leq n$ with $\ell(\oln_l) > 0$. 

\begin{lemma}\label{Lemma - structural}
Suppose $D$ is a dense open subset of $\R(\U)$ and $p \in \R(\U)$. Then there are $p^*_D = \la d_1^*,\dots, d_n^*\ra \geq^* p$, a finite sequence of integers, $1 \leq i_1 < \dots < i_m \leq n$, and a sequence of trees $\la T_1,\dots, T_m\ra$, where each $T_l \subset [V_{\kappa(\oln_{i_l})}]^{\leq n_l}$ is a $\olm(d^*_{i_l})$-fat tree, satisfying the following condition: For every sequence of maximal branches $\la \vec{\eta}_l \mid 1 \leq l \leq m\ra$ with each $\vec{\eta}_l$ maximal in $T_l$, there exists a sequence of sequences of sets $\la \vec{A_l} \mid 1 \leq l \leq m\ra$ such that for every $l$, $\vec{A_l}$ is a $\vec{\eta}_l$-measure-one sequence, and
$p^*_D \fr \la \vec{\eta}_1,\vec{A_1} \ra \fr \la \vec{\eta}_2,\vec{A_2}\ra \fr \dots \fr \la \vec{\eta}_m,\vec{A_m} \ra$ belongs to $D$. 
\end{lemma}

\cite{GitikHB} utilizes  Lemma \ref{Lemma - structural} to prove the following result, originally due to Mitchell \cite{Mitchell}.
\begin{theorem}[Mitchell]\label{Theorem - Regularity}
If $\otp(\ell(\U)) \geq \kappa^+$ then $\kappa$ is regular in any $\R(\U)$ generic extension.
\end{theorem}

\begin{remark}\label{Remark - Cof kappa}
It is also shown in \cite{GitikHB} that the result of Theorem \ref{Theorem - Regularity} is optimal in the sense that $\kappa$ becomes singular in all $\R(\U)$ generic extensions when $\ell(\U) < \kappa^+$. 
A similar argument shows that if $\U$ does not contain a repeat point (see the definition below) and that $\cf(\ell(\U)) \leq \kappa$, then $\kappa$ becomes singular in the generic extension.
\end{remark}

\begin{definition}
A measure $U_\rho \in \U$ is a repeat point if $U_\alpha \subset \bigcup_{i < \rho}U_i$ for every $\alpha \geq \rho$. We say that $\U$ satisfies the \textbf{Repeat Property} ($\RP$) if it contains a repeat point. 
\end{definition}
\begin{theorem}[Mitchell]
If $\U$ satisfies $\RP$ then $\kappa$ remains measurable in a $\R(\U)$ generic extension.
\end{theorem}

We conclude this Section with a proposition concerning fresh subsets of $\kappa$ in a Radin/Magidor generic extensions. Although it will only be used in the last part of the paper (i.e., Lemma \ref{Lemma - Pi11}), we include it here as we believe it is of an independent interest. 

\begin{definition}[Joel Hamkins]
Let $V[G]$ be a generic extension of $V$. A set $X \subset \kappa$ in $V[G]$ is fresh if $X \cap \alpha \in V$ for every $\alpha < \kappa$. 
\end{definition}

The following result is originally due to Cummings and Woodin (\cite{CumWoo}). 

\begin{proposition}\label{Proposition - fresh sets}
Let $\rad(\U)$ be a Radin or a Magidor forcing on a cardinal $\kappa$. 
If the forcing $\R(\U)$ does not change the cofinality of $\kappa$ to $\omega$ then it does not add fresh subsets to $\kappa$. 
\end{proposition}
\begin{proof}
Let $\tau$ be a name of a subset of $\kappa$, such that $0_{\rad}$ forces $\tau \cap \beta \in V$ for every $\beta < \kappa$.
We introduce the following terminology to prove that $\tau$ must coincide with a set $S \in V$. 
For a condition $p = \vec{d} \fr \la\U,A\ra$, $\vec{d} = \la d_i \mid i < k\ra$, let $\supp(p) = \{ \kappa(d_i) \mid i < k\}$, $\kappa_0(p)= \max(\supp(p))$, and $\beta(p)$ denote the $\R(\U)$ name of the minimal ordinal on the Radin generic closed unbounded set $C_G$, which is above $\kappa_0(p)$.
We call the condition $p = \vec{d} \fr \la U,A\ra$ good, if there exists $S \subset \kappa$ in $V$ such that 
$p \force \tau \cap \beta(p) = \can{S} \cap \beta(p)$. 
We denote the set $S$ by $S_p$. Let us first show that the set of good conditions is dense in $\R(\U)$.
Fix $p = \vec{d} \fr \la \U,A\ra$ in $\rad(\U)$. For every $\oln\in A$ of order $0$ (i.e., $\oln = \la \kappa(\oln) \ra$) then $p \fr \la \oln \ra$ has an extension $q = \vec{d}(\oln)\fr \la \oln \ra \fr \la \U, B(\oln)\ra$ forcing $\tau \cap \kappa(\oln) = s(\oln)$ for some $s(\oln) \subset \kappa(\oln)$ in $V$. Note that $\vec{d}$ and $\vec{d}(\oln)$ must have the same maximal ordinal $\kappa(d_{k-1})$. In particular $\vec{d}(\oln) \in V_{\kappa_0(p) + 1}$. 
Next, set  $\vec{d} = [\vec{d}(\oln)]_{U_0}$, $B = \Delta_{\oln}B(\oln)$, 
and $S_p = [s(\oln)]_{U_0}$, then there exists  $A(0) \in U_0$ such that for each $\oln \in A(0)$, $\vec{d} = \vec{d}(\oln)$, $s(\oln) = S_p \cap \kappa(\oln)$, and $B(\oln) \subset B \setminus V_{\kappa(\oln)}$.  Let $A^*$ be the set obtained from $A \cap B$, by reducing the order $0$ measure sequences to $A(0)$. Then $p^* = \vec{d} \fr \la \U, A^* \ra$ is good. \\
Let  $G \subset \rad(\U)$ be a generic filter and suppose  $\tau_G \neq S$ for every set $S \subset \kappa$ in$V$.  Working in $V[G]$, we define an increasing sequence of good conditions $\la p^n \mid n  < \omega\ra$
in $G$. Let $p^0 \in G$ be a good condition and denote $S_{p^0}$ by $S^0$.
Given $p^n \in G$ and $S^n = S_{p^n}$, we use the fact $\tau_G \neq S^n$ to find $p^{n+1} \geq p^n$ in $G$, such that $\kappa_0(p^{n+1}) \cap \tau_G \neq \kappa_0(p^{n+1}) \cap S^n$. We may also 
assume  $p^{n+1}$ is good and set $S^{n+1} = S_{p^{n+1}}$. Clearly, $\kappa_0(p^{n+1}) > \kappa_0(p^{n})$, 
$S^{n+1} \cap \kappa_0(p^n) = S^{n} \cap \kappa_0(p^n)$, and
$S^{n+1} \cap \kappa_0(p^{n+1}) \neq S^n \cap \kappa_0(p^{n+1})$.
Next, let $\gamma = \bigcup_{n<\omega} \kappa_0(p^n)$. Since $\cf(\kappa)^{V[G]} > \omega$, 
  $\gamma \in C_G$. 
By the construction of the conditions $p^n \in G$, $\tau_G \cap \gamma \neq S^n \cap \gamma$ for all $n < \omega$. Let $q \in G$ and $X \subset \gamma$ in $V$ such that
$\gamma \in \supp(q)$ and $q \force \tau \cap \can{\gamma} = \can{X} \cap \can{\gamma}$.
Let us write $q = \vec{d^0} \fr \vec{d^1} \fr \la \U, A_q\ra$, where $\vec{d^0} = \la d^0_i \mid i \leq k^0\ra$, and 
$\kappa(d^0_{k^0}) = \gamma$. 
Take $n < \omega$ so that $\kappa_0(p^n) > \supp(q) \cap \gamma$, and $q^n \in G$ be the minimal common extension of $p^n$ and $q$. Note that $\max(\supp(q^n) \cap \gamma) = \kappa_0(p^n)$.
Pick $\oln \in a_{k^0}(q)$ with $\ell(\oln) = 0$, so that $X \cap \kappa(\oln) \neq S^n \cap \kappa(\oln)$
and consider the extension $q^n \fr \la  \oln \ra$ of $q^n$. 
By the choice of $q^n \geq p^n$,   $q^n \fr \la \oln\ra \force \kappa(\oln) = \beta(p^n)$, 
thus $q^n \fr \la \oln\ra \force \tau \cap \can{\kappa(\oln)} = S^n \cap \can{\kappa(\oln)}$. 
This is an absurd as $q^n \fr \la \oln \ra \geq q$ and therefore  $q^n \fr \la \oln\ra \force \tau \cap \can{\kappa(\oln)} = \can{X} \cap \can{\kappa(\oln)} \neq S^n \cap \can{\kappa(\oln)}$. 
\end{proof}

\section{Radin forcing and stationarity of ground model sets}\label{Section 2}
We utilize Lemma \ref{Lemma - structural} to determine which subsets of $\kappa$ remain stationary in a generic extension by $\R(\U)$.
It is known that if $\U$ is a $\mo$-increasing sequence of measures of length $\ell(\U) < \kappa$ such that $\cf(\ell(\U))$ is uncountable,  then $\kappa$ becomes singular of uncountable cofinality in a Magidor forcing extension by $\U$, and for every $X \subset \kappa$ in $V$,  $X$ remains a stationary subset of $\kappa$ in a generic extension if and only if $X \in U_\tau$ for closed unbounded many $\tau < \ell(\U)$. 
This characterization of ground model sets which remain stationary does not apply to $\R(\U)$ when $\otp(\ell(\U)) \geq \kappa^+$. 

\begin{definition}
\begin{enumerate}
\item Let $Z \subset \MS$ in $V$. We say that $Z$ is $\U$-positive if $Z \in U_\tau$ for unbounded many ordinals $\tau < \ell(\U)$. 
\item For a set $Z \subset \MS$ we define $O(Z) = \{ \kappa(\olm) \mid \olm \in Z\}$. 
\end{enumerate}
\end{definition}

\begin{proposition}\label{Prop - Stationary 1}
Suppose that $cf(l(\U)) \geq \kappa^+$. Then for every $Z \subset \MS$ in $V$, if $Z$ is $\U$-positive then $O(Z)$ is stationary in $V[G]$.
\end{proposition}
\begin{proof}
Let $\tau$ be a $\R(\U)$-name for a closed unbounded subset in $\kappa$. 
We show that every condition $p$ has an extension forcing $O(\can{Z}) \cap \tau \neq \can{\emptyset}$. 
For a condition $q = q_0 \fr \la \U,B\ra$ where $q_0 = \la d_i \mid i < k\ra$, we denote $\supp(q_0) = \{\kappa(d_i) \mid i <k\}$ and $\kappa_0(q) = \max(\supp(q_0))$. 
For every $i < \kappa$, let $D_i \subset \R(\U)$ be the dense open set of all conditions $q = q_0 \fr \la \U, B\ra$ such that $q \force \name{\beta}_i < \kappa_0(q)$, where $\name{\beta}_i$ is the name of the $i-$th element of $\tau$. 
By Lemma \ref{Lemma - structural}, for each $\vec{d} \in \R_{<\kappa}$ there is a sequence of fat trees $\la T_1,\dots, T_m\ra$ associated with $\vec{d}$ and $D_i$. Following the notations of the Lemma, let $T_{i,\vec{d}}$ denote the top tree $T_m$ if it is a $\U$-fat tree, and $A_{i,\vec{d}} \in \bigcap \U$ be the top measure-one set in the condition $p^*_{D_{i}}$. 
Since $cf(l(\U)) \geq \kappa^+$, and there are at most $\kappa$ many trees of the form $T_{i,\vec{d}}$, there exists some $\alpha^*  <\ell(\U)$ which is greater than the indices of all measures associated with the splitting levels of the fat trees $T_{i,\vec{d}}$, $i,\vec{d} \in V_\kappa$.
Define $ \Gamma = \{\oln \in MS \mid \forall i,\vec{d} \in V_{\kappa(\oln)}. \ T_{i,\vec{d}} \cap V_{\kappa(\oln)} \text{ is a } \oln \text{-fat tree }\} $.
It follows that $\Gamma \in \bigcap_{\gamma \geq \alpha^*}U_\gamma$, in particular there exists some $\gamma \geq \alpha$ such that $Z \in U_\gamma$. 
Define $A^* = \triangle_{i,\vec{d}}A_{i,\vec{d}}$ and $p^* = p_0 \fr \la \U, A^* \ra$. \\

\textbf{Claim*:} For every $\oln \in \Gamma \cap Z \cap A^*$,  $p^* \fr \la \oln \ra \force \kappa(\oln) \in \tau$.\\
It is sufficient to verify $p^* \fr \la \oln \ra \force \kappa(\oln) = \beta_{\kappa(\oln)}$. Clearly, $q \force \name{\beta}_{\kappa(\oln)} \geq \kappa(\oln)$, and since $p \force \tau \text{ is closed unbounded}$, it is actually sufficient to show that $p^* \fr \la \oln \ra \force \name{\beta}_{i} < \kappa(\oln)$ for all $i < \kappa(\oln)$.
Fix $i < \kappa(\oln)$, and  $r \geq p^* \fr \la \oln \ra$. We claim $r$ has an extension which forces that $`` \beta_i < \kappa(\oln)``$. 
Suppose $r= r_0 \fr \la \oln, b \ra \fr r_1 \fr \la \U,A_r\ra$. Our construction of $p^*$ guarantees 
$r_0 \in V_{\kappa(\oln)}$ and that  every measure sequence $\oln$ in $r_1$ belongs to $A_{i,r_0}$. Moreover, as $\oln \in \Gamma$, $T_{i,r_0} \cap V_{\kappa(\oln)}$ is a fat $\oln-$tree.
Let $T_1,\dots, T_m$ be the sequence of fat trees associated with $\vec{d} = r_0$.
For every sequence of maximal branches $\mathfrak{t} = \la \vec{\eta}_l \mid 1 \leq l \leq m\ra$ through $\la T_1,\dots,T_m\ra$ respectively,  there is a  sequence of sequences of sets $\mathfrak{a}= \la \vec{A_l} \mid 1 \leq l \leq m\ra$, 
such that the extension
$(r_0 \fr \la \oln,b\ra \fr \la \U,A_{i,r_0} \ra) \fr \la \vec{\eta}_1,\vec{A_1} \ra \fr \la \vec{\eta}_2,\vec{A_2}\ra \fr \dots \fr \la \vec{\eta}_m,\vec{A_m} \ra$ 
belongs to $D_i$. Denote the last condition by $r_0^{+(\mathfrak{t},\mathfrak{a})}$. 
Since $T_{i,\vec{d}} \cap V_{\kappa(\oln)}$ is $\oln$-fat, there is a sequence of maximal branches $\mathfrak{t} = \la \vec{\eta}_l \mid 1 \leq l \leq m\ra$ consisting only of $V_{\kappa(\oln)}$ elements, resulting in a condition $r_0^{+(\mathfrak{t},\mathfrak{a})}$ which is compatible with $r$.  
Finally, as $\kappa_0(r_0^{+(\mathfrak{t},\mathfrak{a})}) = \kappa(\oln)$ and  $\kappa_0(r_0^{+(\mathfrak{t},\mathfrak{a})}) \in D_i$, 
$r_0^{+(\mathfrak{t},\mathfrak{a})} \force \name{\beta_i} < \kappa(\can{\oln})$.
Claim* and the Proposition follow.
\end{proof}

As an immediate corollary of the Lemma, we obtain the following 
result of Woodin.

\begin{corollary}[Woodin]\label{Corollary - Mahloness}
For every $\tau \leq \kappa^+$, If $\otp(l(\U))$ is the ordinal exponent $(\kappa^+)^{1+\tau}$ then $\kappa$ is $\tau$-Mahlo in a $\R(\U)$ generic extension.
\end{corollary}


The next result shows that assuming the sequence $\U$ does not contain a repeat point, the sufficient condition given in Proposition \ref{Prop - Stationary 1} for $O(Z)$ to be stationary is also necessary.

\begin{proposition}\label{proposition - Stationary 2}
Suppose $\U$ is a measure sequence of limit length which does not contain a repeat point. For every $Z \subset \MS$, if 
$Z \in \bigcap (\U\setminus \tau) = \bigcap\{ U_{\rho} \mid \tau \leq \rho < \ell(\U)\}$ for some $\tau < \ell(\U)$ then $O(Z)$ contains a closed unbounded set in a $\R(\U)$ generic extension.
\end{proposition}
\begin{proof}
Fix $Z,\tau < \ell(\U)$ as in the statement of the Lemma. We show that for every $p = p_0 \fr \la \U,A^p\ra \in \R(\U)$ there is a direct extension $p^* = p_0 \fr \la \U,A^*\ra$ forcing that $O(Z)$ contains a closed unbounded set. 
  Since $\tau$ is not a repeat point, there exists $B \in U_\tau \setminus (\bigcup \U\uhr \tau)$.
  Defining $B' = \{ \olm \in \MS \mid \exists i < \ell(\olm). B \cap V_{\kappa(\olm)} \in \olm(i) \setminus (\cup \olm\uhr  i)\}$,  it is routine to verify that for every $\rho < \ell(\U)$, $\U\uhr \rho \in j(B')$ if and only if $\rho > \tau$.  
  By replacing $Z$ with $Z \cap B$ we may assume $Z \in U_\rho$  only for $\rho > \tau$.
  Next, let $Z^{\leq} = \{ \olm \in \MS \mid Z \cap V_{\kappa(\olm)} \not\in \cup\olm\}$. It follows that  $Z^{\leq} \in U_\rho$ if and only if $\rho \leq \tau$. 
We define $A^* = A^p \cap (Z \bigcup Z^{\leq})$, $p^* = p_0 \fr \la \U,A^* \cap A\ra$, and 
$\name{D} = (\name{\MS_G} \setminus \max(p_0)) \cap Z$.  Let us show $p^* \force O(\name{D}) \text{ is closed unbounded in } \kappa$. 
Let $q = \vec{d} \fr \la \U,B\ra$ be an extension of $p^*$, and $\alpha < \kappa$. 
Since $Z$ is unbounded in $V_\kappa$, $q$ has a one point extension $q \fr \la \oln\ra$ where $\oln \in Z \setminus V_\alpha$. Thus $q \fr \la \oln \ra \force \kappa(\oln) \in O(\name{D}) \setminus \alpha$. Finally, suppose $\alpha < \kappa$ and $q \force \alpha \text{ is a limit point of } O(\name{D})$. 
We may assume $\alpha = \kappa(\oln)$ for some $\oln = \oln(d_i)$ for some $d_i \in \vec{d}$. 
Since $\kappa(\oln) > \max(p_0)$, $\oln \in A^* \subset Z^{\leq} \cup Z$. 
$\oln$ cannot be an element of $Z^{\leq}$ as otherwise, $Z \cap V_\alpha \not\in \cup \oln$ and by substituting $a(d_i)$  with $a(d_i) \setminus Z$, we can form a direct extension $q^* \geq^* q$ forcing $\alpha = \kappa(\oln)$ is not a limit point of  $O(\name{D})$. Contradiction. It follows that $\oln \in Z$, and $q \force \kappa(\oln) \in O(\name{D})$. 
\end{proof}

\section{Stationary reflection and the failure of diamond}\label{Section 3}

Woodin's construction of a model of set theory satisfying $\neg\Diamond_\kappa$ on a Mahlo cardinal $\kappa$, is based on the following result.

\begin{theorem}[Woodin]\label{Theorem - Woodin Diamond}
Suppose that $\U$ is a measure sequence and $2^\kappa > \ell(\U)$, and let $G \subset \R(\U)$ be generic over $V$. If $\kappa$ remains regular in $V[G]$ then $\neg\Diamond_\kappa$ holds in $V[G]$. 
\end{theorem}
We include Woodin's elegant argument for completeness. 
\begin{proof}
Let $\name{\vec{s}} = \la \name{s_\alpha} \mid \alpha < \kappa\ra$ be a $\R(\U)$-name, and suppose that $p = p_0 \fr \la \U,A^p\ra$ is a condition forcing $\name{s_\alpha} \subset \alpha$ for all $\alpha < \kappa$. 
For each $\oln \in A^p$, the forcing $\R(\U)/(p \fr \la \oln\ra)$ factors into $\R(\oln)/(p_0 \fr \la \oln\ra) \times R(\U)/(\U,A^p \setminus V_{\kappa(\oln)+1})$, where the direct extension order of the second component is $(2^{\kappa(\oln)})^+$-closed. It follows that the condition $\la \U,A^p \setminus V_{\kappa(\oln)+1})$ in second component has a direct extension $\la \U, A_{\oln}\ra$ which decides the value of the set $\name{s_{\kappa(\oln)}}$, hence reducing the $\R(\U)$ name $\name{s_{\kappa(\oln)}}$ to a $\R(\oln)$-name $\name{s'_{\oln}}$. 
Now, let $A^* = \Delta_{\oln \in A^p}A_{\oln}$. Consider the condition $p^* = p_0 \fr \la \U,A^*\ra$ 
and the function $\vec{s'} : A^* \to V$, defined by $\vec{s'}(\oln) = \name{s'_{\oln}}$.
It follows that $p^* \fr \la \oln\ra \force \name{s_{\kappa(\oln)}} = \vec{s'}(\oln)\}$ for each $\oln \in A^*$. 
Consequently, for every $\tau < \ell(\U)$, $j(p^*) \fr \la \U\uhr \tau\ra \force j(\vec{s})_{\kappa} = j(\vec{s'})(\U\uhr \tau)$, where the last is a $\R(\U\uhr\tau)$ name of a subset of $\kappa$. Since $\R(\U\uhr\tau)$ satisfies $\kappa^+$.c.c and $2^\kappa > \ell(\U)$, there must exist $X \subset \kappa$ such that $j(p^*) \fr \la \U\uhr \tau\ra \force j(\vec{s'})(\U\uhr \tau) \neq \can{X}$ for every $\tau < \ell(\U)$. It follows that $p^*$ has a direct extension $q = p_0 \fr \la \U,B\ra$ such that
$q \fr \la \oln\ra \force \name{s}_{\kappa(\oln)} \neq \can{X}\cap \kappa(\oln)$ for every $\oln \in B$. 
Hence $q$ forces $\name{\vec{s}}$  is not a  $\Diamond_\kappa$  sequence.  
\end{proof}

Woodin's argument essentially implies that every large cardinal property of $\kappa$, obtainable in a $\R(\U)$ generic extension from assumptions concerning the length of $\U$, is consistent with $\neg \Diamond_\kappa$. 
Therefore, from Theorem \ref{Theorem - Regularity} and Corollary \ref{Corollary - Mahloness}, we can infer $\neg\Diamond_\kappa$ is consistent when $\kappa$ is inaccessible, or $\tau$-Mahlo for some $\tau \leq \kappa^+$. Indeed, it is well-known that under certain hypermeasurability large cardinal assumptions 
we can construct a model $V$ in which $2^\kappa = \kappa^{++}$ and $\kappa$ carries a measure sequence $\U$ of length $\kappa^+$, or $(\kappa^+)^{1+\tau}$ for $\tau \leq \kappa^+$.\footnote{I.e., using Mitchell's version of Radin forcing \cite{Mitchell - CUF}, the assumption of a measurable cardinal with $o(\kappa) = \kappa^{++} + (\kappa^+)^{1+\tau}$ suffices.}
By extending the analysis of the stationary subsets of $\kappa$ in $\R(\U)$ generic extensions we prove the following result. 

\begin{theorem}\label{Theorem - SR} 
Let $\U$ be a measure sequence . If $\cf(\ell(\U)) \geq \kappa^{++}$ then $\kappa$ satisfies the strong simultaneous reflection principle in every $\R(\U)$ generic extension. 
\end{theorem}



The following family of functions play a central role in the analysis of stationary subsets of $\kappa$ in a Radin generic extension.

\begin{definition}\label{Definition - MF}
A \textbf{measure function} is a function $b : \MS \to V_\kappa$ satisfying $b(\olm) \in \cap\olm$ for every $\olm \in \MS$. We denote the set of measure functions by $\MF$. 
\end{definition}

\textbf{Proof.}(Theorem \ref{Theorem - SR}) 
First, if $\U$ contains a repeat point then $\kappa$ is measurable in any $\R(\U)$ generic extension, and in particular satisfies the strong simultaneous reflection property. Therefore, let us assume from now on that $\U$ does not contain a repeat point.
We commence with showing 
that every stationary subset $S$ of $\kappa$ in $V[G]$ reflects.
Let $\name{S}$ be a $\R(\U)$ name of $S$. 
Working in $V$, for each $\vec{d} \in \R_{<\kappa}$ consider the condition $p_{\vec{d}} = \vec{d} \fr \la \U,\MS \setminus \max(\vec{d})\ra$. For each $\olm \in \MS \setminus \max(\vec{d})$, the condition $p_{\vec{d}} \fr \la \olm\ra$ has a direct extension of the form $q_{\vec{d}}(\olm) = e_{\vec{d}}(\olm) \fr \la \olm, b_{\vec{d}}(\olm)\ra \fr \la \U, A_{\vec{d}}(\olm)\ra$ deciding the statement $``\kappa(\olm) \in \name{S}``$. 
Therefore, for each $\vec{d}$, we obtain three functions, $e_{\vec{d}}$, $b_{\vec{d}}$, and $A_{\vec{d}}$. 
 Let $A_{\vec{d}} = \Delta_{\olm \in \MS} A_{\vec{d}}(\olm)$
and $A = \Delta_{\vec{d}} A_{\vec{d}}$. Define $b^* \in \MF$ by
$b^*(\olm) = \Delta_{\vec{d}}V_{\kappa(\olm)} b_{\vec{d}}(\olm) = \{ \oln \in \MS\cap V_{\kappa(\olm)} \mid \forall \vec{d} \in V_{\kappa(\oln)}. \oln \in b_{\vec{d}}(\olm)\}$. 
While independent of $\vec{d} \in \R_{<\kappa}$, $A,b^*$ capture the information given by the sets $A_{\vec{d}}$ and measure functions $b_{\vec{d}}$.
Next, for an element $\e  \in \R_{<\kappa}$ 
define $Z_{\e} = \{ \olm \in \MS \mid \exists A \in \bigcap \U. \quad \e \fr  \la \olm, b^*(\olm) \ra \fr \la \U,A\ra \force \kappa(\olm) \in \name{S}\}$.
We say that $\e$ is a stationary witness of $\name{S}$ if there exists $B_{\e} \in \bigcap \U$ such that for every $\vec{\eta} \in \MS^{<\omega}$, $\vec{\eta} \subset B_{\e}$ the set 
$Z_{\e} \dhr \vec{\eta} = \{ \olm \in Z_{\e} \mid \vec{\eta} \subset b^*(\olm) \text{ and }  b^*(\olm) \cap V_{\kappa(\olm')} \in \cap \olm' \text{ for every } \olm' \in \vec{\eta}\}$ is $\U$-positive.\\

\textbf{Claim 1:}
Suppose $p = p_0 \fr \la \U,A_p\ra \in \R(\U)$ forces that $\name{S}$ is a stationary subset of $\kappa$. Then $p$ has an extension $q = \e \fr \la \U, A'\ra$ such that $\e$ is a stationary witness of $\name{S}$ and $A' \subset B_{\e}$.\\
By replacing $p$ with a direct extension if needed,  we may assume that $A_p \subset A$. 
It is sufficient to show that for every generic filter $G \subset \R(\U)$ which contains $p$, there exists $q =  \e \fr \la \U, A_{\e}\ra \in G$ such that $\e$ is a stationary witness of $\name{S}$.
To this end, work in $V[G]$ and for each $\alpha \in S \cap C_G$ set $\olm_\alpha \in \MS_G$ to be the unique measure sequence in $\MS_G$ satisfying $\alpha = \kappa(\olm_\alpha)$. 
For each $\alpha \in S$, $b^*(\olm_\alpha) \in \bigcap \olm_\alpha$. Therefore there exists a maximal ordinal $\beta_\alpha < \alpha$ greater or equal to $\max(\kappa(p_0))$
such that $\MS_G \setminus V_{\beta_\alpha}  \subset b(\olm_\alpha)$. 
By F\`{o}dor's Lemma, there exist $\beta^*  \in C_G$ and $S' \subset S$ stationary, so that $\beta^* = \beta_\alpha$ for each $\alpha \in S'$. 
Let $G_{<\kappa}$ denote the set of bottom parts of generic conditions.\footnote{Namely, the set of all $\vec{d} \in \R_\kappa$ such that $\vec{d} \fr \la \U,\MS\ra \in G$.} 
By a standard density argument, for each $\alpha \in S'$ there exists $\vec{d}^\alpha =\la d^\alpha_i \mid i < k_\alpha\ra$ extending $p_0$, with $\kappa(d^\alpha_{k_\alpha-1}) = \beta^*$, such that $e_{\vec{d}^\alpha}(\olm_\alpha) \in G_{<\kappa}$. By pressing down again, we can find $\vec{d^*},\e \in V_{\beta^*+1}$, and a stationary set $S^* \subset S'$ such that $\vec{d^*} = \vec{d}^\alpha$ and $\e = e_{\vec{d}^*}(\olm_\alpha)$ for each $\alpha \in S^*$. 
It follows that for each $\alpha \in S^*$, the condition $\e \fr \la \olm_\alpha,b^*(\olm_\alpha) \ra \fr \la \U,A_p\ra$ belongs to $G$ and forces that $``\alpha \in \name{S}``$. 
This implies $\{ \olm_\alpha \mid \alpha \in S^*\} \subset Z_{\e}$, which in turn implies that $Z_{\e}$ is $\U$-positive.  To see the last, note that otherwise, 
Proposition \ref{proposition - Stationary 2} implies there is a closed unbounded set $C \subset \kappa$ in $V[G]$ which is disjoint from $O(Z_{\e}) = \{ \kappa(\olm) \mid \olm \in Z_{\e}\}$. But this is an absurd as $S^* \cap C \neq \emptyset$. \\
Our next goal is to construct a set $B_{\e}$ as described in the definition of $\e$ being a stationary witness of $\name{S}$. For this we consider sets of the form $Z_{\e}\dhr \vec{\eta}$ for various $\vec{\eta} \in A^{<\omega} \subset \MS^{<\omega}$.
Let us say that $\vec{\eta}$ is $0$-positive if $Z_{\e}\dhr \vec{\eta}$ is $\U$-positive, and for an integer $n \geq 0$, say $\vec{\eta}$ is $(n+1)$-positive if the set $B^n_{\vec{\eta}}= \{ \oln \in \MS \mid \vec{\eta} \fr \la \oln\ra \text{ is } n\text{-positive}\}$ belongs to $\bigcap \U$. 
Finally, $\vec{\eta}$ is $\omega$-positive if it is $n$-positive for each $n < \omega$ and $B^\omega_{\vec{\eta}} = \bigcap_{n < \omega} B^n_{\vec{\eta}}$.
Let $B^{\omega} = \{ \oln \in \MS \mid \la \oln \ra \text{ is } \omega\text{-positive}\}$. \\
\textbf{Sub-Claim 1.1:} 
$B^\omega \in \bigcap \U$. 
It is sufficient to show that for every $n  <\omega$, the set
$B^n =  \{ \oln \in \MS \mid \la \oln \ra \text{ is } n\text{-positive}\}$ belongs to $\bigcap\U$. 
Suppose otherwise.  Then there are $\alpha_0 < \ell(\U)$ and $A_0 \in U_{\alpha_0}$ such that $\la \oln_0 \ra$ is not $n$-positive for each $\oln_0 \in A_0$. That is, for each $\oln_0 \in A_0$ there are $U_{\alpha_{\la \oln_0 \ra}}$ and $A_{\la \oln_0 \ra} \in U_{\la \oln_0\ra}$ such that for each $\oln_1 \in A_{\la \oln_0\ra}$, $\la \oln_0,\oln_1\ra$ is not $(n-1)$-positive. By continuing to unravel the statement in this manner, we can construct a $\U$-fat tree $T \subset \MS^{\leq n}$ (see Definition \ref{Definition - fat trees}) such that for every maximal branch $\vec{\eta} = \la \oln_1,\dots, \oln_{n}\ra$ of $T$, $Z_{\e} \dhr {\vec{\eta}}$ is not $\U$-positive. Since $T$ is $\U$-fat, a standard density argument shows there exists a condition $e^* \fr \vec{d} \fr \la \U,A'\ra$ in $G$ such that $\vec{d} =\la d_1,\dots,d_n\ra$, where $\vec{\eta} = \la \olm(d_1),\dots,\olm(d_n)\ra$ is a maximal branch of $T$. Now by Proposition \ref{proposition - Stationary 2}, there exists a closed unbounded set $C \subset \kappa$ in $V[G]$ which is disjoint from $O(Z_{\e} \dhr {\vec{\eta}})$. 
To get a contradiction, we take $\alpha \in S^* \cap C$ which is above $\max(\kappa(\oln_n))$. By the definition of $S^*$, we have that $e^* \fr \la \olm_\alpha,b^*(\olm_\alpha)\ra \fr \la \U,A^*\ra$ belongs to $G$ and therefore forces $``\can{\alpha} = \can{ \kappa(\olm_\alpha)} \in \name{S}``$. 
Also, since the ordinals in $\vec{\eta}$ are all above $\max(\e) = \beta^*$, $\vec{\eta} \subset b^*(\olm_\alpha)$. But this means $\alpha \in O(Z_{\e}\dhr {\vec{\eta}}) \cap C$. Contradiction.  \qed{(Sub-Claim 1.1)}\\
${}$
We can now define $B_{\e}$. First, let  $\Delta^1 = B^\omega$ and for each $n < \omega$, let $\Delta^{n+1} = \Delta^n \cap \{ \olm \in \Delta^n \mid \forall \vec{\eta} =  \la \olm_1,\dots,\olm_n \ra  \subset \Delta^n \cap V_{\kappa(\olm)}. \olm \in B^{\omega}_{\vec{\eta}}\}$. 
We then set $B_{\e}$ to be $A\cap (\bigcap_n \Delta^n)$. It is routine to verify $B_{\e} \in \bigcap\U$ and that for every increasing sequence $\vec{\eta} = \la \olm_1,\dots,\olm_m\ra \subset B_{\e}$, $Z_{\e}\dhr \vec{\eta}$ is $\U$-positive. 
Finally, let $A' = B_{\e} \cap A_p$.  
Then $\e$ is stationary witness of $\name{S}$ and $q= \e \fr \la \U, A'\ra$ is an extension of $p$.
\qed{(Claim 1)}\\

Let us show how a stationary witness of $\name{S}$, $\e \in \R_{<\kappa}$, can be used to find a reflection point of $\name{S}$.
Let $B_{\e} \in \bigcap \U$ as in the definition of a stationary witness.  For each $\vec{\eta} \in B_{\e}^{<\omega}$ define $E_{\e}(\vec{\eta}) \subset \ell(\U)$ to be the set of accumulation points of all $\tau < \ell(\U)$ such that $Z_{\e}\dhr \vec{\eta} \in U_\tau$. 
Since each $E_{\e}(\vec{\eta})$ is closed unbounded in $\ell(\U)$ and $\cf(\ell(\U)) \geq \kappa^{++}$,  $E_{\e} = \bigcap\{ E_{\e}(\vec{\eta}) \mid \vec{\eta} \subset A_{\e}\}$ is also closed unbounded and there exists $\tau \in \kappa^{++} \cap \Cof(\kappa^+)$ which is a limit point of $E_{\e}$. 
It follows that there exists $X \in U_\tau$ such that every $\oln \in X$ satisfies the following conditions: 
\begin{enumerate}
\item $\cf(\ell(\oln)) = \kappa(\oln)^+$,
\item for every $\vec{\eta} \subset A_{\e} \cap V_{\kappa(\oln)}$, $Z_{\e}\dhr \vec{\eta}$ is $\oln$-positive.
\end{enumerate}

\textbf{Claim 2:}
For every $\oln \in X$, the condition $\e \fr \la \oln, A_{\e} \cap V_{\kappa(\oln)}\ra \fr \la \U,A_{\e}\ra$ forces $\name{S} \cap \kappa(\oln)$ is stationary in $\kappa(\oln)$. \\
Let us denote the condition $\e \fr \la \oln, A_{\e} \cap V_{\kappa(\oln)}\ra \fr \la \U,A_{\e}\ra$ by $t$. Suppose that $\sigma$ is a name for a subset of $\kappa(\oln)$ and $q$ is an extension of $t$, forcing that $\sigma$ is a closed unbounded subset of $\kappa(\oln)$.
We separate $q$ into parts and write $q = q_0 \fr q_1 \fr \la \oln, b\ra \fr q_2 \fr \la \U,A_q\ra$, where $q_0 \geq \e$, $q_1 \fr \la \oln,b\ra \geq \la \oln,A_{\e} \cap V_{\kappa(\oln)}\ra$, and $q_2 \fr \la \U,A_q\ra \geq \la \U,A_{\e}\ra$.
By further extending $q_2 \fr \la \U,A_q\ra$ if necessary, we may assume  $\sigma$ is a $\R(\oln)$ name of a closed unbounded subset of $\kappa(\oln)$.
The rest of the proof follows the argument of the proof of Proposition \ref{Prop - Stationary 1}, applied to the forcing $\R(\oln)$. For each $i < \kappa(\oln)$ and $\vec{d} \in R_{<\kappa(\oln)}$ we define a $\oln$-fat tree
$T_{i,\vec{d}}$ and a $\U$ set $A_{i,\vec{d}}$, associated with the set $D_i$ of all $\R(\oln)$ conditions $r = r_0 \fr \la \oln,a_r\ra$ which force the $i$-th element of $\sigma$ to be bounded in $\kappa_0(r) = \max(r_0)$. 
We then define $ \Gamma = \{\olm \in \MS\cap V_{\kappa(\oln)} \mid \forall i,\vec{d} \in V_\kappa(\olm). \ T_{i,\vec{d}} \cap V_{\kappa(\olm)} \text{ is a fat-}\olm \text{ tree }\} $.
Since $\cf(\ell(\oln)) \geq \kappa(\oln)^+$, there exists $\alpha^* < \ell(\oln)$ so that  $\Gamma \in \bigcap_{i \geq \alpha^*}\oln(i)$. 
Let $\vec{\eta} \in \MS^{<\omega}$ be an increasing enumeration of the measure sequences in $q_1$. \footnote{Namely, if $q_1 = \la d_1,\dots,d_k\ra$ then $\vec{\eta} = \la \olm(d_1),\dots,\olm(d_k)\ra$.} Since $q$ extends $t$,  $\vec{\eta} \subset A_{\e} \cap V_{\kappa(\oln)}$, and by our assumption $\oln \in X$,  $Z_{\e}\dhr \vec{\eta}$ must be a $\oln$-positive. 
Hence, there must exist $\olm \in (Z_{\e}\dhr \vec{\eta}) \cap b$
such that $T_{i,\vec{d}} \cap V_{\kappa(\olm)}$ is $\olm$-fat for each $i,\vec{d} \in V_{\kappa(\olm)}$. 
By Claim* of Proposition \ref{Prop - Stationary 1}, 
$q \fr \la \olm\ra$ forces $\kappa(\olm) \in \sigma$. Furthermore, the fact $\olm \in Z_{\e}\dhr \vec{\eta}$ implies $q \fr \la \olm\ra$ is compatible with the condition $\e \fr \la \olm ,b^*(\olm)\ra \fr \la \U,A\ra$, which forces $\kappa(\olm) \in \name{S}$. Hence $q$ has an extension which forces $\sigma \cap \name{S} \neq \emptyset$. \qed{(Claim 2)}\\

Claims 1,2 imply that if $p = p_0 \fr \la \U,A_p\ra$ is a condition which forces $\name{S}$ is a stationary subset of $\kappa$, then $p$ has an extension 
of the form $\e \fr \la \oln, B_{\e} \cap V_{\kappa(\oln)}\ra \fr \la \U,B_{\e}\ra$ forcing that  $\name{S} \cap \kappa(\oln)$ is stationary. It follows that in a $\R(\U)$ generic extension $V[G]$, every stationary subset of $\kappa$ reflects.\\

For the final part of the proof we extend the argument to obtain the strong simultaneous reflection property at $\kappa$.
Suppose $\la \name{S}_i \mid i < \kappa\ra$ is a sequence of names of subsets of $\kappa$ and $p = p_0 \fr \la \U,A_p\ra$ is a condition of $\R(\U)$ forcing that each $\name{S_i}$ is a stationary in $\kappa$. 
For each $i < \kappa$ let $W(\name{S_i})$ denote the set of all $\e \in \R_{<\kappa}$ which are stationary witnesses of $\name{S_i}$.
As shown above, for each $\e \in W(\name{S_i})$
there exists $B^i_{\e} \in \bigcap\U$ and a closed unbounded set $E^i_{\e} \subset \ell(\U)$, such that for every limit point $\tau \in E^i_{\e}$ of cofinality $\kappa^+$, 
there exists a set $X \in U_\tau$ which consists of $\oln$ for which the condition
$\e \fr \la \oln, B_{\e} \cap V_{\kappa(\oln)}\ra \fr \la \U,B_{\e}\ra$ forces 
$\name{S_i}$ reflects at $\kappa(\oln)$. 
For each $i < \kappa$, define $A^i = \Delta_{\e \in W(\name{S_i})} B^i_{\e} = \{ \oln \in \MS \mid \forall \e \in W(\name{S_i}) \cap V_{\kappa(\oln)}. \medskip \oln \in B^i_{\e}\}$ and 
$E^i = \bigcap_{\e \in W(\name{S_i})} E^i_{\e}$. 
Finally, define   $A^* = \Delta_{i < \kappa} A^i$ and $E^* = \bigcap_{i < \kappa} E^i$. 
We conclude that there exists a set $X \subset \MS$ which belongs to each $U_\tau$ where $\tau$ is a limit point of $E^*$ of cofinality $\kappa^+$, 
such that for every $\oln \in X$, $i < \kappa(\oln)$, and $\e \in W(\name{S_i}) \cap V_{\kappa(\oln)}$, the condition
$\e \fr \la \oln, A^* \cap V_{\kappa(\oln)}\ra \fr \la \U,A^*\ra$ forces $\name{S} \cap \kappa(\oln)$ is stationary in $\kappa(\oln)$.
Note that $X$ is $\U$-positive. Let $G \subset \R(\U)$ be a generic filter containing $p ^* = p_0 \fr \la \U,A^*\ra$. By Proposition \ref{Prop - Stationary 1}, the set $O(X) = \{ \kappa(\oln) \mid \oln \in X\}$ is a stationary subset of $\kappa$ in $V[G]$. 
For each $i < \kappa$, let $S_i = (\name{S_i})_G$. By Claim 1 above, $W(\name{S_i}) \cap G_{<\kappa} \neq \emptyset$. Let $\e_i$ be the lexicographic minimal sequence in $ W(\name{S_i}) \cap G_{<\kappa}$, and $\kappa(\e_i)$ denote its maximal critical point. In $V[G]$, define $f : \kappa \to \kappa$ in $V[G]$ by $f(i) = \kappa(\e_i) + 1$. 
Since $O(X)$ is stationary, there exists $\oln \in \MS_G \cap X$ such that $\alpha = \kappa(\oln)$ is a closure point of $f$. It follows that for each $i < \alpha$, 
$\e \fr \la \oln, A^* \cap V_{\kappa(\oln)}\ra \fr \la \U,A^*\ra$ belongs to $G$, hence $S_i \cap \alpha$ is a stationary subset of $\alpha$. \qed{(Theorem  \ref{Theorem - SR})}



\section{Weak compactness and Radin forcing}\label{Section 4}
It is natural to ask whether the Radin forcing machinery can be extended to establish the consistency of $\neg\Diamond_\kappa$ at a weakly compact cardinal $\kappa$. 
One necessary step required towards giving an affirmative answer to this question, is to find a reasonably weak assumption of a measure sequence $\U$ which implies $\kappa$ is weakly compact in a $\R(\U)$ generic extension. The section will be mostly devoted to 
providing a property of $\U$, called the weak repeat property ($\WRP$),  which characterizes weak compactness of $\kappa$ in a $\R(\U)$ generic extension. 
In the last part of the section, we return to the violation of the diamond question and discuss some natural obstructions raised by the weak compactness characterization.

\begin{definition}
${}$
\begin{enumerate}
\item We say that a filter $W \subset \power(\MS)$ 
measures a set $X \subset \MS$ if $X \in W$ or $\MS \setminus X \in W$. If $F \subset \power(\MS)$ is a family of sets, then we say $W$ measures $F$ if it measures each $X \in F$.
For every $b \in \MF$ and $\olm \in \MS$ let $X_{b,\olm} =\{ \oln \in \MS \mid \olm \in b(\oln)\}$. 
We say that a filter $W$ measures a function $b \in \MF$ if it measures the family $\F_b = \{X_{b,\olm} \mid \olm \in \MS\}$. Whenever $W$ measures $b \in \MF$,  we define $[b]_W = \{ \olm \in \MS \mid X_{b,\olm} \in W\} \subset \MS$. 

\item 
Let $b \in \MF$ and $W \subset \power(\MS)$ be a filter. We say that $W$ is a \textbf{repeat filter} of $b$ with respect to $\U$ if it satisfies the following conditions.\\
\textbf{a.} $W$ is a $\kappa$-complete filter extending the co-bounded filter on $\MS$,\footnote{Namely, for every $\alpha < \kappa$, the set $\{ \olm \in \MS \mid \kappa(\olm) > \alpha\} \in W$.}\\
\textbf{b.} $W \subset \bigcup \U$,\\
\textbf{c.} $W$ measures $b$,\\
\textbf{d.} $[b]_W \in \cap \U$.

\item We say that $\U$ satisfies the \textbf{Weak Repeat Property} ($\WRP$) if every $b \in \MF$ has a repeat filter with respect to $\U$.  
\end{enumerate}
\end{definition}

Let us say $\U$ satisfies the repeat property ($\RP$) if it contains a repeat point measure. 

\begin{lemma}\label{Lemma-RPvsWRP}
$\RP$ implies $\WRP$. Moreover, if $U_\rho$ is a repeat point of $\U$ then 
$\{ \olm \in \MS \mid \olm \text{ satisfies } \WRP \}$ belongs to $U_\rho$ and there exists $\tau < \rho$ such that $\U\uhr \tau$ satisfies $\WRP$.
\end{lemma}

\begin{proof}
Let $\rho$ be the first repeat point on $\U$. Then $\U\uhr \rho$ does not satisfy $\RP$.  Nevertheless, $\bigcap \U \uhr \rho = \bigcap\U$ and so $\kappa$ remains regular (and even measurable) in a generic extension by $\R(\U\uhr \rho) = \R(\U)$. By Remark \ref{Remark - Cof kappa}, it follows that $\cf(\rho) \geq \kappa^+$. 
To establish the first assertion, note that $W = U_\rho$ is a repeat filter of every $b \in \MF$. 
Indeed, $[b]_{U_\rho} \in \bigcap \U\uhr \rho = \bigcap\U$. 
Next, we claim that for each $b \in \MF$ there exists $\tau < \rho$ such that $U_\tau$ is a repeat filter of $b$ with respect to both $\U$ and $\U\uhr \rho$. Fix $b \in \MF$ and an enumeration $\la Y_i \mid i < \kappa\ra$ of $\F_b = \{ X_{b,\olm} \mid \olm \in \MS\}$.
For each $i < \kappa$ let 
\[Y_i' = 
\begin{cases}
Y_i &\mbox{ if }  Y_i \in U_\rho\\
\MS \setminus Y_i &\mbox{ otherwise } 
\end{cases}\]
Let $Y' = \Delta_{i<\kappa} Y_i'$. $Y'  \in U_\rho$ since $U_\rho$ is normal.
Since $U_\rho$ is a repeat point there exists some $\tau < \rho$ such that
$Y' \in U_\tau$. It follows that $[b]_{U_\tau} = [b]_{U_\rho} \in \bigcap (\U\uhr \rho)$, and thus $W = U_\tau \in \U\uhr \rho$ is a repeat filter of $b$. As these witnesses are known to  $M$,  $
M\models \U\uhr \rho \text{ satisfies } \WRP$, and $\{ \olm \in \MS \mid \olm \text{ satisfies } \WRP \} \in U_\rho$. The fact $U_\rho$ is a repeat point implies there exists $\tau < \rho$ such that 
$\{ \olm \in \MS \mid \olm \text{ satisfies } \WRP \} \in U_\tau$, which in turn, implies $M \models \U\uhr \tau \text{ satisfies } \WRP$. Since $\MF \subset M$, it follows that $\U\uhr \tau$ satisfies $\WRP$ in $V$. 
\end{proof}

\begin{theorem}\label{Theorem-WRPvsWC}
$\kappa$ is weakly compact in a $\rad(\U)$ generic extension if
and only if $\U$ satisfies the Weak Repeat Property.
\end{theorem}




\subsection{From the Weak Repeat Property to Weak Compactness}

Suppose that $\U \in V$ is a measure sequence on $\kappa$, satisfying $\WRP$. 
Let $G \subset \rad(\U)$ be a generic filter over $V$. 
To show $\kappa$ weakly compact in $V[G]$, it is sufficient to prove that for every sufficiently large regular cardinal $\theta > \kappa$ and  $N' \elem H_\theta[G]$ satisfying ${}^{<\kappa} N' \subset N'$, $G,\U\in N'$, and $|N'| = \kappa$, there exists a $\kappa$-complete $N'$-ultrafilter $U'$ on $\kappa$. 
That is, $U'$ measures all the sets in $\power(\kappa) \cap N$ and is closed under intersection of sequences of its elements of length less than $\kappa$.
Since $\rad(\U)$ satisfies $\kappa^+$.c.c, $N'$  has an elementary extension of the form $N[G]$ (i.e., $N' \elem N[G] \elem H_\theta[G]$) for some $N \elem H_\theta$ in $V$, such that 
$|N|  = \kappa$, $N^{<\kappa} \subset N$, and $\U \in N$.
We therefore focus on models $N[G]$ of this form.

\begin{lemma}\label{Lemma - measure Models}
Let $\theta > \kappa$ be a regular cardinal, and $N \elem H_\theta$ be an elementary 
substructure of cardinality $\kappa$ with  $V_\kappa \subset N$.
If $\U$ satisfies $\WRP$ then there exists a 
$\kappa$-complete filter $W \subset \bigcup \U$, which measures all the subsets of $\MS$ in $N$ and all $b \in N \cap \MF$, and which satisfies $[b]_W \in \bigcap \U$ for every $b \in  N \cap \MF$.
\end{lemma}

\begin{proof}
Fix an enumeration $\la b_i \mid i < \kappa\ra$ of $\MF \cap N$. 
Define $b' \in \MF$ by $b'(\olm) = \triangle_{i < \kappa(\olm)}b_i(\olm) = \{\oln \in V_{\kappa(\olm)} \mid \forall i < \kappa(\oln) \oln \in b_i(\olm)\}$.
It follows that for every filter $W$, if $W$ is a repeat filter of $b'$ then it measures
each $b_i$ and $[b_i]_W \supset [b']_W \setminus V_{i+1} \in \bigcap \U$. Therefore if $W$ is a repeat filter of $b'$ then it is also a repeat filter of each $b_i$, $i < \kappa$.
Next, we tweak $b'$ to obtain $b^* \in \MF$ such that every filter $W$ which measures $b^*$ also measures $\power(\MS) \cap N$. 
Let $\{A_i \mid i < \kappa\}$ be an enumeration of $\power(\MS) \cap N$ and fix an auxiliary set $X \subset \MS$ such that $|X| = \kappa$ and $O(X) \cap \rho$  is nonstationary in $\rho$ for every regular cardinal $\rho \leq \kappa$.
Therefore, any modification in the measure function $b'$ which is restricted to
$X$ will not affect its key properties of $b'$ established above.
Fix an enumeration $\{ \olm_i \mid i < \kappa\}$ of $X$ and define $b^* : \MS \to V_\kappa$ as follows. For every $\oln \in \MS$ let
$b^*(\oln) = (b'(\oln) \setminus X) \uplus \{\olm_i \in V_{\kappa(\oln)} \mid \oln \in A_i\}$. 
Clearly, $b^*(\oln) \setminus X = b'(\oln) \setminus X \in \bigcap \oln$ for each $\oln \in \MS$, thus $b^* \in \MF$. Furthermore, for each $i < \kappa$, 
$A_i \setminus  V_i = \{ \oln \mid \olm_i \in b^*(\oln)\}$.
It follows that if $W$ is a repeat filter of $b^*$ then $W$ is a repeat filter of $b'$ and it measures all the sets $A_i \in \power(\MS) \cap N$. 
\end{proof}

Let $N \elem H_\theta$ such that ${}^{<\kappa}N \subset N$ and $\U \in N$, and fix a repeat filter $W \subset \power(\MS)$ given by Lemma \ref{Lemma - measure Models}. Working in $V[G]$, we define an $N[G]$-filter $U_W$.
\begin{definition}
Let $U_W$ be the set of all $X \in \power(\kappa) \cap N[G]$, for which there exists a name $\name{X}\in N$ such that $X = \name{X}_G$, and there are $p  = p_0 \fr \la \U,A_p\ra \in G$ and 
$b \in \MF \cap N$ such that 
\begin{itemize}
\item $A_p \subset [b]_W$, and
\item $\{\olm \in \MS  \mid \exists A(\olm) \in \bigcap\U. \medskip p_0 \fr \la \olm,b(\olm) \ra \fr \la \U,A(\olm)\ra \force \kappa(\olm) \in \name{X}\} \in W$.
\end{itemize}
Given $\name{X},p,b$ as in the definition, we say $p$ and $b$ \textbf{witness} $X \in U_W$, and denote the set $\{\olm \in \MS \mid \exists A(\olm) \in \bigcap\U. p_0 \fr \la \olm,b(\olm) \ra \fr \la \U,A(\olm)\ra \force \kappa(\olm) \in \name{X}\}$ by $Z(\name{X},p,b)$. 
\end{definition}

Note that the definition of $Z(\name{X},p,b)$ depends only on $p_0,b,\name{X} \in N$. This implies that the set $Z(\name{X},p,b)$ belongs to $N$ and thus is measured by $W$.  
The following two Lemmata show that $U_W$ is a $\kappa$-complete $N[G]$ ultrafilter. 

\begin{lemma}\label{U_W basics}
${}$\\
\textbf{1.} Suppose $p,b$ witness $X \in U_W$. Then for every $q \geq p$ there exists some
$b' \in \MS \cap N$ so that $q,b'$ witness $X \in U_W$ as well .\\
\textbf{2.} If $X,Y \in N[G] \cap \power(\kappa)$ with $X \in U_W$ and $X \subset Y$, then $Y \in U_W$.
\end{lemma}
\begin{proof}${}$\\
\textbf{1.} Let $\name{X} \in N$ be a $\R(\U)$-name of $X$ such that $Z(\name{X},p,b) \in W$.
Given $q \geq p = p_0 \fr \la \U, A_p\ra$, we split $q$ into three parts,
$q = q_0 \fr q_1 \fr \la \U,A_q\ra$, where $q_0 \geq p_0$
and $q_1 \fr \la \U,A_q\ra \geq \la \U,A_p\ra$. 
We have that $A_q \subset A_p \setminus \max(q_1) \subset [b]_W \setminus \max(q_1)$, and note that since $q_1 \in N$, the set $Z = \{ \olm \in \MS \mid q_1 \fr \la \mu, b(\olm) \setminus \max(q_1)\ra \geq \la \olm, b(\olm)\ra\}$ belongs to $N$. Therefore $Z$ is measured by $W$, and furthermore, since $A_p \subset [b]_W$ and $\supp(q_1) \subset A_p$, $Z$ must be a member of $W$. 
Define a function $b' \in \MF$ by setting $b'(\oln)$ to be $b(\oln) \setminus V_{\max(q_1)+1}$ if $ \kappa(\oln) > \max(q_1)$, and $b(\oln)$ otherwise.
It follows that $b' \in \MF \cap N$, $[b']_W = [b]_W \setminus \max(q_1)$, and 
$A_q \subset [b']_W$. We conclude that  for each $\olm \in Z(\name{X},p,b) \cap Z$, $q_0 \fr q_1 \fr \la \olm, b'(\olm)\ra \geq p_0 \fr \la \olm,b(\olm)\ra$, hence, 
by the definition of $Z(\name{X},p,b)$, there exists  $A(\olm) \in \bigcap\U$ so that 
$q_0 \fr q_1 \fr \la \olm, b'(\olm)\ra \fr \la \U, A(\olm)\ra \force \can{\kappa(\olm)} \in \name{X}$.
As the last applies to every $\olm \in Z(\name{X},p,b) \cap Z \in W$, where  $Z \cap Z(\name{X},p,b) \in W$, we conclude $q,b'$ witness $X \in U_W$. \\
%
%
\textbf{2.} Suppose $p,b$ witness $X = \name{X}_G \in U_W$ and  $\name{Y} \in N$ is a name such that
$X \subset \name{Y}_G $.  Since $\R(\U)$ satisfies $\kappa^+.c.c$ and $|N| = \kappa$, and $\kappa \subset N$, 
there must exist some
 $t \in N \cap G$ forcing $\name{X} \subset \name{Y}$.
Writing $t = \vec{t_0} \fr \la \U,A_t\ra$ we have that $A_t \in \bigcap\U \cap N$ must belong to $W$ 
and 
$t,p \in G$ must be compatible. Let $q \geq p,t$ be a common extension in $G$, and let
$b \in \MS \cap N$ so that $q,b$ witness $X \in U_W$ via $\name{X}$, namely, $Z(\name{X},q,b) \in W$. 
For every $\olm \in Z(\name{X},q,b)$, there exists some $A(\olm) \in \bigcap\U$
such that $q_0 \fr \la \olm,b(\olm)\ra \fr \la \U,A(\olm)\ra \force \can{\kappa(\olm)} \in \name{X}$.
Furthermore, if $\olm \in A_t \setminus \max(q_0)$ then 
$q_0 \fr \la \olm,b(\olm) \fr \la \U, A(\olm) \cap A_t\ra$ is an extension
of $t$ and forces $\kappa(\olm) \in \name{Y}$.
It follows that $Z(\name{X},q,b) \cap A_t \subset Z(\name{Y},q,b)$, thus $Z(\name{Y},q,b) \in W$. 
We conclude that $q,b$ witness $Y \in U_W$
\end{proof}

It follows from the first part of Lemma \ref{U_W basics} that $U_W$ is closed under intersections of its sets, and by the second part of the Lemma that it is upwards closed under inclusions. Hence, $U_W$ is a filter on $\power(\kappa) \cap N[G]$. It remains to show that it is $\kappa$-complete. 
We first introduce the following terminology.
\begin{definition}${}$\\
\textbf{1. }
Let $D \subset \rad(\U)$ be a dense set. We say $D$ is \textbf{strongly dense} if for every $p \in \R(\U)$, $p = p_0 \fr \la \U, A\ra$, 
there exists some $q \in D$, $q \geq p$ such that $q = q_0 \fr \la \U, A'\ra$, and $\kappa(q_0) = \kappa(p_0)$ (i.e. $q_0 \geq p_0$ in $\rad_{<\kappa}$). \\

\noindent 
\textbf{2. }
Let $\vec{D} = \la D_\nu \mid \nu < \kappa\ra$ be a sequence of strongly dense sets, and $p_0 \in \rad_{<\kappa}$.
Define three functions $b_{p_0,\vec{D}}$, $B_{p_0,\vec{D}}$, $r_{p_0,\vec{D}}$ with domain $\MS$: Fix some well ordering of $V_{\kappa+1}$ and 
consider the condition $p = p_0 \fr \la \U, \MS \setminus V_{\kappa(p_0)}\ra$. 
    For every $\oln \in \MS \setminus V_{\kappa(p_0)}$, Let $q$ be the first extension of $p \fr \oln$ which belongs to $D_{\kappa(\oln)}$.
Writing $q = r' \fr \la \oln, a' \ra \fr \la \U, A'\ra$, we set $b_{p_0,\vec{D}}(\oln) = a'$,  $B_{p_0,\vec{D}}(\oln) = A'$, and $r_{p_0,\vec{D}}(\oln) = r'$. 
Since $N \elem H_\theta$, it follows that for every sequence of strongly dense sets $\vec{D} \in N$ and $p_0 \in  \R_{<\kappa} \subset N$, 
$b_{p_0,\vec{D}},B_{p_0,\vec{D}},r_{p_0,\vec{D}}$ all belong to $N$ as well. 
\end{definition}

\begin{lemma}
Let $\lambda < \kappa$ and suppose that  $\la X_i \mid i < \lambda\ra$ is a partition of $\kappa$ in $N[G]$. Then there exists $i^* < \lambda$ such that $X_{i^*} \in U_W$.
\end{lemma}

\begin{proof}
Since $N^{<\kappa} \subset N$, there is a sequence of names $\la \name{X_i} \mid i < \lambda \ra$ in $N$ such that $X_i = (\name{X_i})_G$ for each $i < \lambda$.
The claim will follow from a density argument once we show that for every $p = p_0 \fr \la \U,A_p\ra  \in \R(\U)$, there are $r \geq p$, $i^* < \lambda$, and $b \in \MS \cap N$, such that
$r,b$ witness $X_{i^*} \in U_{W}$.\\
For every $\nu < \kappa$ let $D_\nu = \{p' \in \rad(\U) \mid \exists i < \lambda. p' \force \nu \in \name{X_i}\}$.
Each $D_\nu$ is strongly dense,\footnote{Every $p \in \R(\U)$ has a direct extension in $D_\nu$.} and $\vec{D} = \la D_\nu \mid \nu < \kappa\ra$ belongs to $N$.
Let $b_{p_0,\vec{D}}, r_{p_0,\vec{D}}, B_{p_0,\vec{D}} \in N$ be the associated functions defined above.
For each $i < \lambda$, let $Z_i = \{\olm \in \MS \mid r_{p_0,\vec{D}}(\olm) \fr \la \olm, b_{p_0,\vec{D}}(\olm)\ra \fr \la \U , B_{p_0,\vec{D}}(\olm)\ra \force \can{\kappa(\olm)} \in \name{X_i}\}$.
As the sets $Z_i$, $i < \lambda$, are pairwise disjoint and belong to $N$, there exists a unique $i^* < \lambda$ such that $Z_{i^*} \in W$.
Furthermore, since $W$ is $\kappa$-complete and measures $N$, there exists $r_0 \geq p_0$ such that  $\{\olm \in Z_{i^*} \mid r_{p_0,\vec{D}}(\olm) = r_0\} \in W$.
Define $A_r = A_p \cap [b_{p_0,\vec{D}}]_W \cap \Delta_{\olm \in Z_{i^*}}B_{p_0,\vec{D}}(\olm)$, and 
$r = r_0 \fr \la \U, A_r\ra$. Then $r \geq p$ and $A_r \subset [b]_W$, where $b = b_{p_0,\vec{D}}$ is in $N$. Furthermore,
for every $\olm \in Z_{i^*}$,
$r_0 \fr \la \olm , b(\olm) \ra \fr \la \U, B_{p_0, \vec{D}}\ra \force \can{\kappa(\olm)} \in \name{X_i^*}$. 
It follows that $Z_{i^*}  \subset Z(\name{X_i^*},r,b)$, and thus $Z(\name{X_i^*},r,b) \in W$. 
\end{proof}

\subsection{From Weak Compactness to the Weak Repeat Property}

Let $G \subset \R(\U)$ be a generic filter.
Recall $G$ is completely determined by its induced sequence of measure sequences,  $\MS_G =  \{ \olm \in \MS \mid \exists p = \la d_i \mid i \leq k\ra \in G. \olm = \olm(d_i) \text{ for some } i < k\}$. 

Suppose  $\kappa$ is weakly compact in $V[G]$, and fix a measure function $b$ in $V$. We would like to show $b$ has a repeat filter $W$ in $V$. If $\U$ satisfies $\RP$ there is noting to show. We therefore assume $\U$ does not contain a repeat point. 
Then, by Remark \ref{Remark - Cof kappa},  $\cf(\ell(\U))$ must be at least $\kappa^+$ for $\kappa$ to be weakly compact in a generic extension.

To accomplish this, we construct a $\Pi^1_1$ statement $\varphi$ of the structure $M_b = \la V_\kappa[G],\in,b,V_\kappa,\MS_G\ra$ such that $M_b \models \phi$ if and only if $b$ does not have a repeat filter in $V$, and show that the reflections of $\phi$ to $\alpha < \kappa$ fail on a closed unbounded set of cardinals $\alpha < \kappa$. Since $\kappa$ is weakly compact, it follows that 
$M_b$ must satisfy $\neg \varphi$, thus $b$ has a repeat filter. \\

We commence by observing that the existence of a repeat filter for $b$ is witnessed by a family of $\kappa$ many subsets of $\power(\MS)$. Recall that for every $b \in \MF$, we define $\F_b = \{X_{b,\olm} \mid \olm \in \MS\}$, where
for each $\oln \in \MS$, $X_{b,\olm} =\{ \oln \in \MS \mid \olm \in b(\oln)\}$. Clearly $|\F_b| = \kappa$. 
\begin{definition}
A subset $P$ of $\power(\MS)$ is called a \textbf{repeat Prefilter} of $b$ (with respect to $\U$) if it satisfies the following properties: \\
\textbf{a.} For every $\lambda < \kappa$ and every sequence $\la X_i \mid i < \lambda \ra \subset P$, the intersection $\bigcap_{i<\lambda} X_i \in \bigcup \U$. \\
\textbf{b.} $P \subset \F_b \cup \{ \MS \setminus X \mid X \in \F_b\}$. \\
\textbf{c.} $P$ measures $b$
. In particular, 
$[b]_P = \{ \olm \in \MS \mid X_{b,\olm} \in P\}$ is defined. \\
\textbf{d.} $[b]_P \in \bigcap \U$. 
\end{definition}

It is easy to see that if $W$ is a repeat filter of $b$ then $P = W \cap \left( \F_b \cup \{ \MS \setminus X \mid X \in \F_b\} \right)$ is a prefilter of $b$, and that if $P$ is a prefilter of $b$ then its upwards closure $W = \{Y \subset \MS \mid \exists X \in P. X \subset Y\}$ is a repeat filter of $b$.

\begin{definition}
Working in $V[G]$, let $\varphi$ be the following statement: \\
For every $P \subset \power(\MS)$ of cardinality $\kappa$, at least one of the following conditions hold.\\
\textbf{$\varphi_1$.} $P \not\in V$\\
\textbf{$\varphi_2$.} There exists $\lambda < \kappa$ and a sequence $\la X_i \mid i < \lambda \ra \subset P$ such that 
$\bigcap_{i<\lambda} X_i \not\in \bigcup \U$\\
\textbf{$\varphi_3$.} $P \not\subset \F_b \cup \{ \MS \setminus X \mid X \in \F_b\}$\\
\textbf{$\varphi_4$.} $P$ does not measure $b$\\
\textbf{$\varphi_5$.} $P$ measures $b$ and $[b]_P \not\in \bigcap\U$. 
\end{definition} 

It is clear that $M_b \models \varphi$ if and only if $b$ has a repeat prefilter. 

\begin{lemma}\label{Lemma - Pi11}
$\varphi$ is equivalent to a $\Pi^1_1$ statement over $M_b = \la V_\kappa[G],\in,b,V_\kappa,\MS_G\ra$. 
\end{lemma}
\begin{proof}
Since any family $P \subset \power(\MS)$ of size $\kappa$ can be enumerated as a subset of $\MS \times \kappa$, we identify $P \subset \MS \times \kappa$ with a sequence $\la X_i \mid i < \kappa\ra$, where $X_i = \{ \oln \in \MS \mid (\oln,i) \in P\}$.
$\varphi$ is clearly equivalent to a statement of the form $\forall P \subseteq (\MS \times \kappa).  (\varphi_1 \vee \varphi_2 \vee \varphi_3 \vee \varphi_4 \vee \varphi_5)$.  It is therefore sufficient to verify each $\varphi_i$ is equivalent to a $\Sigma^0_\omega$ statement over $M_b$. We take each $\varphi_i$ at a time. \\
$(\varphi_1).$  By Proposition \ref{Proposition - fresh sets}, $\R(\U)$ does not add fresh subsets to $\kappa$, and hence, neither to $\MS \times \kappa$. Therefore, $\varphi_1$ is equivalent to $\exists\alpha < \kappa. P \cap (V_\alpha \times \alpha) \not\in V_\kappa$ which is clearly equivalent to a $\Sigma^0_\omega$ statement over $M_b$. 
$(\varphi_2).$ An easy density argument shows that for every $A \subset \MS$ in $V$, $A \in \bigcup \U$ if and only if $A \cap \MS_G$ is not bounded in some $V_\alpha$, $\alpha < \kappa$. Therefore, $\varphi_2$ is equivalent to $\exists \lambda < \kappa \exists \alpha < \kappa. (\bigcap_{i < \lambda}X_i \setminus V_\alpha) = \emptyset$, which is clearly equivalent to a $\Sigma^0_\omega$ statement over $M_b$. 
$(\varphi_3 + \varphi_4)$. It is straightforward to verify $\varphi_3$ and $\varphi_4$ are equivalent to $\Sigma^0_\omega$ statement over $M_b$, using the fact $b \subset \MS \times \MS$ is a part of the augmented structure $M_b$.
$(\varphi_5).$ It is easy to see that the first part of the assertion, `` $P$ measures $b$``, is equivalent to a $\Sigma^0_\omega$ statement of $M_b$. Considering the second part, ``$[b]_P \not\in \U$`` of $\varphi_5$,  note that the same density argument used for the description of $\varphi_1$ shows that a $V$ set $A \subset \MS$ belongs to $\bigcap \U$ if and only if $\MS_G$ is almost contained in $A$. Therefore the second part of $\varphi_5$ is equivalent to the $M_b$ statement 
$\forall \alpha < \kappa \medskip \exists \olm \in \MS_G \setminus V_\alpha \medskip \exists i < \kappa \medskip \forall \oln \in \MS. \medskip (\oln \in X_i \iff \olm \not\in b(\oln))$.  
\end{proof}


\begin{lemma} In $V[G]$, there exists a closed unbounded set of $\alpha < \kappa$
for which $\la V_\alpha[G],\alpha,b,V_\alpha,\MS_G\ra \models \neg \varphi$. 
\end{lemma}
\begin{proof}
Assuming $\U$ does not contain a repeat point, Proposition
\ref{proposition - Stationary 2} implies it is sufficient to show there exists $\tau < \ell(\U)$ and $Z \in \bigcap (\U\setminus \tau)$, such that 
in $V$, for every $\oln \in Z$, the restriction $b\uhr V_{\kappa(\oln)}$ has a weak repeat filter with respect to $\oln$. 
Equivalently, it is sufficient to check $b = j(b) \uhr V_\kappa$ has a weak repeat filter with respect to $U_\eta$, for every $\eta \in [\tau,\ell(\U))$.
Let $\F_b = \{ X_{b,\olm} \mid \olm \in \MS\}$ and $F^*$ be the family of all intersections of length $\lambda < \kappa$ of sets in 
$\F_b \cup \{\MS \setminus X \mid X \in \F_b\}$. Fix an enumeration $\la Y_i \mid i < \kappa\ra$ of $F^*$, and for each $Y_i$ which does not belong to $\bigcap \U$, let $\tau_i < l(\U)$ be the 
first $\tau < \ell(\U)$ such that $Y_i \not\in U_{\tau}$. Since $\cf(\ell(\U)) \geq \kappa^+$,  $\tau = \sup_{i<\kappa} \tau_i + 1$ is below $\ell(\U)$. 
Fix an ordinal $\eta \in [\tau, l(\U))$. We have that for every $i < \kappa$,
$Y_i \not\in U_\eta$ implies that $Y_i \not\in \bigcap \U$,
which in turn, implies $Y_i \not\in \U\uhr\eta$.
Define a prefilter $P = U_\eta \cap F^*$. Since $\eta \geq \tau$, every intersection of $\lambda < \kappa$ sets of $P$ belongs to $\bigcup (U\uhr \tau) \subset \bigcup (\U\uhr \eta)$. It is also clear that $[b]_P = [b]_{U_\eta} \in \bigcap \U\uhr \eta$.  It follows that in $V$, $P$ is a repeat prefilter for $b$ with respect to $\U\uhr \eta$. 
Considering $j : V \to M$, it is clear that $P,b \in M$ and that $M \models P \text{ is a repeat prefilter with respect to } b$ as well. 
\end{proof}

\subsection{Towards the failure of diamond on a weakly compact cardinal}
In light of Theorem \ref{Theorem - Woodin Diamond} and Theorem \ref{Theorem-WRPvsWC},
the following question is prominent. 

\begin{question}\label{Question - Main}
Is it consistent there exists a measure sequence $\U$ on a cardinal $\kappa$ such that $\U$ satisfies the weak repeat property and $2^\kappa > \ell(\U)$?
\end{question}

We conclude this Section with a discussion describing some of the obstructions to a positive answer to the above question. 

The definition of the weak repeat property (WRP) and the proof of Lemma \ref{Lemma-RPvsWRP}
suggest WRP has a natural (seemingly) stronger property which is still weaker than the existence of a repeat point (RP). 

\begin{definition}
Let us say that a measure sequence $\U$ satisfies the \textbf{Local Repeat Property} ($\LRP$) if for every $b \in \MF$ there exists $\tau < \ell(\U)$ such that $[b]_{U_\tau} \in \bigcap\U$. 
\end{definition}
 
Clearly, $\RP \implies \LRP \implies \WRP$, and the proof of Lemma \ref{Lemma-RPvsWRP} implies that if $\U$ has a repeat point $U_\rho$, then $\{ \olm \in \MS \mid \olm \text{ satisfies } \LRP \} \in U_\rho$. 
Moreover, it is not difficult to see $\LRP$ is equivalent to the variant of $\WRP$ which restricts the  possible repeat filters $W$ to the normal ones. 
\begin{observation}\label{Observation - RP}
Let us say that a measure sequence $\U$ satisfies $\WRP^+$ if every $b \in \MF$ has a repeat filter $W$ which is normal (i.e., closed under diagonal intersections). Then $\WRP^+$ is equivalent to $\LRP$.
\end{observation}
\begin{proof}
Suppose $\U$ satisfies $\WRP^+$, and let $b \in \MF$ and $W$ a normal repeat filter of $b$. Let $P$ be the restriction of $W$ to $\F_b \cup \{ \MS \setminus X \mid X \in F_b\}$ where $\F_b = \{ X_{\oln,b} \mid \oln \in \MS\}$. 
$P$ has size $\kappa$ and measures $b$ and $[b]_P = [b]_W \in \bigcap\U$. Let $X^*$ be a diagonal intersection of the sets in $P$. Then $X^* \in W \subset \bigcup \U$, thus $X^* \in U_\rho$ for some $\rho < \ell(\U)$. Since every $X \in P$ is almost contained in $X^*$, it follows that $[b]_{U_\rho} = [b]_P \in \bigcap \U$. 
\end{proof}


Albeit natural, the $\LRP$ cannot be targeted to provide an affirmative answer to Question \ref{Question - Main}.

\begin{proposition}\label{Proposition - LRP and Cont}
Let $\U$ be a measure sequence on a cardinal $\kappa$.
If $2^\kappa \geq \ell(\U)$ then $\U$ fails to satisfy the local repeat property.
\end{proposition}
\begin{proof}
Let $j : V \to M$ be an embedding which generates the measure sequence $\U$. 
Denote $2^\kappa$ by $\lambda$. 
Let ${x}^\kappa = \la x^\kappa_\alpha \mid \alpha < \lambda\ra$ be an enumeration of $\power(\kappa)$ in $M$.
We may assume there is a sequence $\vec{x} = \la x^\alpha \mid \alpha < \kappa\ra$
so that $x^\alpha$ enumerates $\power(\alpha)$ and $j(\vec{x})(\kappa) = x^\kappa$. 
Since $\lambda \geq \ell(\U)$, $\lambda > \ell(\U\uhr \alpha)$ for every $\alpha < \ell(\U)$, hence the set $A = \{\olm \in \MS \mid 2^{\kappa(\olm)} > l(\olm)\}$ belongs to  $\bigcap \U$. 
Define a measure function $b \in A$ by taking $b(\olm)$ to be the set of all $\oln \in \MS \cap V_{\kappa(\olm)}$ for which $x^{\olm}_{\ell(\olm)} \cap \kappa(\oln) = x^{\oln}_\beta$ for some $\beta > \ell(\oln)$. 
Then for each $\alpha < \ell(\U)$, $[b]_{U_\alpha}$ is the set of all $\oln \in \MS$ so that $x^\kappa_\alpha \cap \kappa(\oln) = x^{\oln}_\beta$ for some $\beta \geq \ell(\oln)$. Denoting this set by $X$, it is easy to see that $\U\uhr \beta \in j(X)$ for every $\beta < \alpha$. Hence $[b]_{U_\alpha} \in \bigcap(\U\uhr\alpha)$.
The same argument, applied to $\olm \in \MS$ instead of $\U\uhr \alpha$ shows that $b \in \MF$.
We claim that this set does not belong to $U_{\alpha}$ 
Indeed, $\U\uhr\alpha \in j(X)$ if and only if $j(x^\kappa_\alpha) \cap \kappa = x^\kappa_\alpha$ is of the form $j(\vec{x})^\kappa_\beta$ for some $\beta > \ell(\U\uhr\alpha) = \alpha$ which is absurd. 
\end{proof}

\textbf{Acknowledgments:} The author would like to thank Thomas Gilton and Mohammad Golshani for many valuable comments and suggestions which greatly improved the manuscript.

\bibliographystyle{plain}
\bibliography{refS}

\end{document}